\documentclass[a4paper,14pt]{amsart}
\usepackage[frenchb]{babel}
\usepackage[utf8]{inputenc}
\usepackage[colorlinks=true]{hyperref}
\hypersetup{urlcolor=green,linkcolor=brown,citecolor=blue,colorlinks=true}

\usepackage{pgf,tikz}
\usepackage{mathrsfs}
\usetikzlibrary{arrows}
\usepackage{amssymb,fancyhdr,txfonts,pxfonts}
\usepackage{graphicx}

\usepackage{graphicx}
\usepackage{pdfsync}
\usepackage[T1]{fontenc}
\usepackage{amsmath}
\usepackage{amssymb}

\newtheorem{defi}{Définition}[section]
      
      \newtheorem{prop}[defi]{Proposition}
      \newtheorem{prop-def}[defi]{Proposition-Définition}
      \newtheorem{def-prop}[defi]{Définition-Proposition}
      \newtheorem{thm}[defi]{Théorème}

      \newtheorem{lem}[defi]{Lemme}
      \newtheorem{cor}[defi]{Corollaire}

\def\geod{{\text {géod}}}

\def\N{{\mathbf N}}

\title[Actions affines isométriques propres des groupes hyperboliques]{Actions affines isométriques propres des groupes hyperboliques sur des quotients d'espaces $\ell^p$}
\author{Aurélien Alvarez et Vincent Lafforgue}
\address{Aurélien Alvarez fait partie du projet ANR-14-CE25-0004 GAMME.\newline
MAPMO, UMR 7349, Université d'Orléans \newline Rue de Chartres, BP 6759 - 45067 Orléans cedex 2, France}
\address{Vincent Lafforgue fait partie du projet  ANR-14-CE25-0012 SINGSTAR.\newline
CNRS et Institut Fourier, UMR 5582, Université Grenoble Alpes \newline 100 rue des Maths, 38610 Gières, France}

\begin{document}

\maketitle

\begin{abstract}
Nous démontrons que tout groupe hyperbolique admet une action affine isométrique propre sur un quotient d'un espace de Banach $\ell^p$, pour tout $p>1$ suffisamment proche de 1.
\end{abstract}


\section{Point de départ}

Dans un travail précédent \cite{alvarez-lafforgue}, nous donnions une nouvelle démonstration élémentaire et auto-contenue d'un théorème de Yu : tout groupe hyperbolique admet une action affine isométrique propre sur un espace $\ell^p$ pour $p$ suffisamment grand \cite{yu}.
Nous renvoyons le lecteur à l'introduction de \cite{alvarez-lafforgue} pour une présentation du contexte mathématique et des principaux résultats autour des actions affines isométriques propres, et au premier paragraphe pour ce qui concerne les définitions ; nous rappelons ici la méthode, bien connue par ailleurs, que nous utilisons pour construire des actions affines.

\vspace{0.2cm}

\noindent
\textbf{Une méthode utile pour construire des actions affines.}
Soit $(\mathcal{E}^{\circ},\| \cdot \|)$ un espace de Banach et $\pi$ une action continue isométrique d'un groupe topologique $G$ sur~$\mathcal{E}^{\circ}$.
Supposons donné un espace vectoriel $V$, un plongement d'espace vectoriel de $\mathcal{E}^{\circ}$ dans~$V$ et une action linéaire de $G$ sur $V$ qui stabilise $\mathcal{E}^{\circ}$ et dont la restriction à $\mathcal{E}^{\circ}$ est~$\pi$.
Si $\xi_{\circ}$ est un élément de $V$ tel que, pour tout $g$ de $G$, le vecteur $\xi_{\circ} - g \cdot \xi_{\circ}$ appartient au sous-espace $\mathcal{E}^{\circ}$ et $g \longmapsto \xi_{\circ} - g \cdot \xi_{\circ}$ est continue de $G$ dans $\mathcal{E}^{\circ}$, alors l'action linéaire de $G$ sur $V$ se restreint en une action continue affine isométrique de $G$ sur l'espace de Banach affine $\mathcal{E} = \xi_{\circ} + \mathcal{E}^{\circ}$.
Cette dernière est aussi donnée par le cocycle $g \longmapsto \xi_{\circ} - g \cdot \xi_{\circ}$.
Par définition, l'action est propre si et seulement si $\lim_{g \to \infty} \| \xi_{\circ} - g \cdot \xi_{\circ} \| = \infty$.

\vspace{0.2cm}

Comme exemple immédiat d'application de cette méthode, nous en déduisons dans \cite{alvarez-lafforgue} que tout groupe discret $G$ de type fini admet une action affine isométrique propre sur l'espace de Banach affine $d(1,\cdot)+\ell^\infty(G)$, le cocycle étant dans ce cas donné par la fonction $c(g) = d(1,\cdot) - d(g,\cdot)$, où $d$ désigne la métrique des mots une fois fixé un système fini de générateurs.

\begin{prop}\label{point-depart}
Tout groupe discret $G$ de type fini admet une action affine isométrique propre sur un quotient d'un espace de Banach $\ell^1$ d'une réunion finie de copies de $G$.
\end{prop}

\begin{proof}
Commençons par rappeler que si $\mathcal{G} = (\mathcal{G}^0,\mathcal{G}^1)$ est un graphe orienté connexe, où $\mathcal{G}^0$ et $\mathcal{G}^1$ désignent les ensembles de sommets et d'arêtes, on a un opérateur {\it bord}
\[\partial : \ell^1(\mathcal{G}^1) \longrightarrow \mathcal{F}({\mathcal{G}^0}),\]
\noindent
où $\mathcal{F}({\mathcal{G}^0})$ désigne l'espace vectoriel des fonctions réelles sur les sommets.
Précisons ce point.
Par définition, $\ell^1(\mathcal{G}^1)$ est la complétion pour la norme $||\cdot||_1$ de l'espace vectoriel $\mathcal{F}_f({\mathcal{G}^1})$ des fonctions à support fini sur les arêtes de $\mathcal{G}$.
L'image par l'opérateur bord d'une arête $x \rightarrow y$ est par définition la 0-chaîne $\delta_y - \delta_x$ qui est une fonction (à support fini) de somme nulle sur les sommets de $\mathcal{G}$.
Par linéarité, $\partial$ s'étend en une application linéaire de $\mathcal{F}_f({\mathcal{G}^1})$ dans l'espace vectoriel $\mathcal{F}^{0}_f({\mathcal{G}^0})$ des fonctions à support fini et de somme nulle sur les sommets.
De plus, cette application est surjective : étant donné deux sommets~$x$ et~$y$ de $\mathcal{G}$, la fonction $\delta_y - \delta_x$ est l'image de n'importe quel chemin d'arêtes reliant~$x$ à $y$.
On peut ainsi définir une norme image sur $\mathcal{F}^{0}_f({\mathcal{G}^0})$ qui se calcule facilement sur une fonction de la forme $\delta_y - \delta_x$ : par définition, c'est l'infimum des normes $||\cdot||_1$ des fonctions $c$ de $\mathcal{F}_f({\mathcal{G}^1})$ telles que $\partial(c) = \delta_y - \delta_x$.
Il n'est pas difficile de voir que pour calculer cet infimum il suffit de considérer les fonctions caractéristiques des chemins d'arêtes reliant~$x$ à $y$ (pour une démonstration complète de ce point, on pourra se reporter au lemme~\ref{chaine-somme-chemins-lacets}), et qu'on obtient ainsi la distance dans le graphe entre~$x$ et~$y$.
Par complétion des espaces de fonctions à support fini, on en déduit finalement que l'opérateur bord s'étend en une application linéaire bornée surjective de $\ell^1(\mathcal{G}^1)$ dans la complétion pour la norme image de $||\cdot||_1$ de $\mathcal{F}^{0}_f({\mathcal{G}^0})$.

Dans le cas où $\mathcal{G}$ est le graphe de Cayley de $G$ associé à un système fini de générateurs, l'action naturelle de $G$ sur $\mathcal{G}^0$ induit une action isométrique $\pi$ de~$G$ sur l'espace de Banach $\mathcal{E}^{\circ}$ obtenu par complétion pour la norme image de $||\cdot||_1$ de l'espace vectoriel $\mathcal{F}^{0}_f({\mathcal{G}^0})$ des fonctions à support fini et de somme nulle sur les sommets.
L'action affine de $G$ sur l'espace affine des fonctions à support fini et de somme~1 sur les sommets induit alors par cette complétion une action affine isométrique de partie linéaire $(\pi,\mathcal{E}^{\circ})$.

Nous allons maintenant expliciter cette action affine isométrique de $G$ dans le cadre de la méthode rappelée ci-dessus pour construire des actions affines.
En prenant $\xi_0 = \delta_1$ la masse de Dirac en l'élément neutre de $G$, $\mathcal{E}^0 = \partial(\ell^1(\mathcal{G}^1))$ comme espace de Banach et $V=\mathcal{F}({\mathcal{G}^0})$ comme espace vectoriel, on retrouve l'action affine isométrique ci-dessus en tant qu'action affine isométrique de $G$ sur l'espace de Banach affine $\delta_1 + \partial(\ell^1(\mathcal{G}^1))$ donnée par le cocycle $g \longmapsto \delta_1 - \delta_g$.
La norme de celui-ci n'est autre que la distance de $g$ à l'élément neutre du groupe, donc l'action est propre.
Il ne reste plus qu'à noter que le $G$-espace des arêtes $\mathcal{G}^1$ s'identifie à une réunion finie de copies de $G$.
En effet, pour tout entier naturel $R$, l'application $(x,y) \longmapsto (x,x^{-1}y)$ de $X^{\leq R}  = \{(x,y)\in G \times G ; d(x,y)\leq R\}$ dans $G \times B(1,R)$ est une bijection $G$-équivariante, où l'action de $G$ sur $G \times B(1,R)$ est par translation à gauche sur le premier facteur\footnote{On désigne par $B(1,R)$ (respectivement $S(1,R)$) la boule (resp. la sphère) de centre~1 et de rayon~$R$.}.
En particulier, $\mathcal{G}^1$ s'identifie à $G \times S(1,1)$.
\end{proof}

Nous verrons dans la suite de cet article comment étendre la proposition~\ref{point-depart} à des quotients d'espaces $\ell^p$ dès lors que $p$ est suffisamment proche de 1 dans le cas des groupes hyperboliques.
Notons dès à présent qu'il est nécessaire de considérer des quotients des espaces $\ell^p$, et non pas les espaces $\ell^p$ eux-mêmes, pour des $p$ proches de 1.
En effet, d'après Nowak \cite{nowak09}, admettre une action affine isométrique propre sur un espace $\ell^p$ pour $1<p<2$ est une caractérisation de la propriété de Haagerup, et l'on sait bien que certains groupes hyperboliques ont la propriété (T) de Kazhdan (par exemple les réseaux de $\text{Sp}(n,1)$ \cite{gromov}).



\section{Énoncé et démonstration du théorème principal}

Nous rappelons la définition d'une classe d'espaces hyperboliques d'un type particulier (contenant les graphes de Cayley des groupes hyperboliques) et nous renvoyons le lecteur au paragraphe~3 de \cite{alvarez-lafforgue} pour les définitions et exemples de base à propos des espaces hyperboliques.

\begin{defi}
Un bon espace hyperbolique discret est un espace hyperbolique, uniformément localement fini dont la métrique provient d'une structure de graphe.
\end{defi}

Notons que puisque la métrique $d : X \longrightarrow \N$ provient d'une structure de graphe, un bon espace hyperbolique discret $(X,d)$ est {\it géodésique} dans le sens que, pour tous $x,y$ de $X$ et tout entier $k$ de $[\![0,\cdots,d(a,b)]\!]$, il existe un élément $z$ de $X$ tel que $d(x,z)=k$ et $d(z,y)=d(x,y)-k$.
Par ailleurs, $d$ étant géodésique, l'uniforme locale finitude équivaut au fait que le nombre de points à distance $1$ d'un élément donné de $X$ est borné indépendamment de ce point.

\vspace{0.1cm}

\noindent
{\bf Exemple fondamental} : si $\Gamma$ est un groupe hyperbolique et $d$ la distance invariante à gauche associée à la longueur des mots pour un système fini de générateurs donné, alors $(\Gamma,d)$ est un bon espace hyperbolique discret ; en outre ce dernier est muni d'une action isométrique de $\Gamma$ par translation à gauche ($d$ provient de la structure de graphe de Cayley sur $\Gamma$ associée au système de générateurs donné).

\vspace{0.1cm}

\noindent
{\it Notations.}
Si $\eta$ est un réel positif et $x,y$ sont deux points de l'espace métrique $X$, on désigne par $\eta\text{-}\text{géod}(x,y)$ l'ensemble des points $z$ de $X$ tels que
$$d(x,z)+d(z,y) \leq d(x,y)+\eta.$$
On note $\geod(x,y)$ au lieu de $0\text{-}\geod(x,y)$ l'ensemble des points $z$ de $X$ tels que $d(x,z)+d(z,y)=d(x,y)$.
Pour tout entier naturel $n$, un chemin de $x$ à $y$ de longueur~$n$ est la donnée d'une suite finie de points $(x_i)_{0 \leq i \leq n}$ de~$X$ telle que $x_0=x$, $x_n=y$ et $d(x_i,x_{i+1})=1$ pour tout~$i$ dans $[\![0,n-1]\!]$ ; si de plus $x_0 = x_n$, nous appellerons un tel chemin un \textit{lacet}.
Enfin, convenons d'appeler \textit{arêtes orientées} les éléments de $X^1  = \{(x,y)\in X \times X ; d(x,y)=1\}$ ; on désigne par $e^{-}$ le sommet origine $x$ d'une telle arête $e=(x,y)$ de $X^1$, par $e^{+}$ son sommet terminal $y$ et par $e^{\text{op}}=(y,x)$ l'arête opposée.

\vspace{0.2cm}

Nous énonçons à présent le théorème principal de cet article.

\begin{thm}\label{thm-principal}
Soit $(X,d)$ un bon espace hyperbolique discret muni d'une action continue isométrique propre d'un groupe topologique $G$.
Il existe un réel $p_0$ strictement supérieur à 1 tel que, pour tout réel $p$ appartenant à $[1,p_0[$, $G$ admet une action continue affine isométrique propre sur un quotient de l'espace de Banach $\ell^p(X^1)$.
\end{thm}

Puisqu'un groupe hyperbolique est en particulier un bon espace hyperbolique discret dont l'espace des arêtes s'identifie à une réunion finie de copies du groupe, le corollaire suivant est une conséquence immédiate du théorème~\ref{thm-principal}.

\begin{cor}
Pour tout groupe hyperbolique $\Gamma$, il existe un entier naturel $k$ non nul tel que, en notant $\Gamma_k$ la réunion disjointe de $k$ copies de~$\Gamma$, pour tout $p$ suffisamment proche de 1, $\Gamma$ admet une action affine isométrique propre sur un quotient de l'espace de Banach $\ell^p(\Gamma_k)$.
\end{cor}

La suite de cet article est consacrée à la démonstration du théorème~\ref{thm-principal} dont la stratégie de preuve reprend celle de la proposition~\ref{point-depart} : le cœur du problème est de démontrer qu'un certain cocycle, analogue au cocycle introduit dans la proposition~\ref{point-depart}, est propre, ce qui n'est pas vrai en général pour $p>1$ sans hypothèse sur l'espace métrique $(X,d)$.
Nous allons utiliser de manière cruciale l'hypothèse d'hyperbolicité qui implique qu'un chemin entre deux points $x$ et $y$ de $(X,d)$ reste proche de toute géodésique entre ces deux points.
Une formulation quantitative de cette heuristique est donnée par la proposition~\ref{prop-principale} dont on voit bien qu'elle tombe en défaut pour un espace métrique comme $(\mathbf{Z}^2,d_{\text{eucl.}})$.

\begin{proof}
Le groupe $G$ opère naturellement dans l'espace vectoriel $\mathcal{F}_f(X^1)$ des fonctions réelles à support fini sur $X^1  = \{(x,y)\in X \times X ; d(x,y)=1\}$, ainsi que dans l'espace vectoriel $\mathcal{F}_f^0(X)$ des fonctions réelles à support fini et de somme nulle sur $X$.
Par ailleurs l'application linéaire \textit{bord} $\partial : \mathcal{F}_f(X^1) \longrightarrow \mathcal{F}_f^0(X)$, définie par $\partial \delta_{(x,y)}=\delta_y-\delta_x$ est surjective et $G$-équivariante.
Pour $p$ dans $]1;\infty[$, on munit l'espace vectoriel $\mathcal{F}_f(X^1)$ des fonctions réelles à support fini sur $X^1$ de la norme $\ell^p$ notée $||\cdot||_p$.
L'action de~$G$ étant isométrique, après complétion des espaces de fonctions ci-dessus, on en déduit une action continue affine isométrique de~$G$ sur l'espace de Banach affine obtenu par complétion pour la norme image de $||\cdot||_p$ par $\partial$ de l'espace affine $\mathcal{F}_f^1(X)$ des fonctions à support fini et de somme 1 sur $X$.
Il reste alors à voir que le cocycle défini par cette action est propre, pour tout $p$ suffisamment proche de 1, ce que nous déduirons de la proposition~\ref{prop2}. 
Pour prouver cette dernière, nous démontrerons d'abord la proposition~\ref{prop-principale} et son corollaire~\ref{cor-prop-principale}.

\vspace{0.3cm}

Soit $(X,d)$ un bon espace hyperbolique discret et $\delta$ une constante d'hyperbolicité de $X$ que l'on supposera strictement positive.
Commençons par démontrer le \textit{lemme d'emboîtement des géodésiques}, lemme bien connu et qui nous sera utile par la suite.

\begin{lem}\label{lemme-emboitement}
Soit $a,b,c,d$ quatre points de~$X$.
Si $b$ appartient à $\eta_1\text{-}\geod(a,c)$, si $c$ appartient à $\eta_2\text{-}\geod(b,d)$ et si $d(b,c)>(\eta_1+\eta_2+\delta)/2$, alors $b$ appartient à $(\eta_1+\delta)\text{-}\text{géod}(a,d)$ et $c$ à $(\eta_2+\delta)\text{-}\text{géod}(a,d)$.
\end{lem}

\begin{proof}
Par définition de la $\delta$-hyperbolicité de $X$ (\cite{alvarez-lafforgue}, déf.~3.2), on a
\[d(a,c)+d(b,d) \leq \text{max}(d(a,d)+d(b,c),d(a,b)+d(c,d)) + \delta.\]
Puisque $b$ appartient à $\eta_1\text{-}\geod(a,c)$ ({\it i.e.} $d(a,b)+d(b,c) \leq d(a,c) + \eta_1$), on en déduit donc
\[d(a,b)+d(b,c)+d(b,d) \leq \text{max}(d(a,d)+d(b,c),d(a,b)+d(c,d)) + \eta_1+ \delta,\]
ou encore
\[d(a,b)+d(b,d) \leq \text{max}(d(a,d),d(a,b)+d(c,d)-d(b,c)) + \eta_1+ \delta.\]
En utilisant à présent que $c$ appartient à $\eta_2\text{-}\geod(b,d)$ ({\it i.e.} $d(b,c)+d(c,d) \leq d(b,d) + \eta_2$), on en déduit alors
\[d(a,b)+d(b,d) \leq \text{max}(d(a,d),d(a,b)+d(b,d)-2d(b,c)+\eta_2) + \eta_1+ \delta.\]
Comme $d(b,c)>(\eta_1+\eta_2+\delta)/2$, on en conclut finalement que
\[d(a,b)+d(b,d) \leq d(a,d)+ \eta_1 + \delta.\]
On démontre de même que $c$ appartient à $(\eta_2+\delta)\text{-}\text{géod}(a,d)$ en appliquant la première partie du lemme que nous venons de démontrer au quadruplet $(d,c,b,a)$.
\end{proof}

Étant donné un point $t$ dans $X$, on note $(\cdot,\cdot)_t$ le produit de Gromov par rapport à $t$, fonction définie sur $X \times X$ par
\[(x,y)_t = \frac{1}{2}(d(x,t)+d(y,t)-d(x,y)).\]
Rappelons que, pour tous points~$x$ et~$y$ de~$X$ et pour tout segment géodésique $\gamma_{xy}$ entre ces deux points, la différence $d(t,\gamma_{xy})-(x,y)_t$ entre la distance de $t$ à $\gamma_{xy}$ et le produit de Gromov de $x,y$ par rapport à~$t$ est, en valeur absolue, majorée par une constante ne dépendant que de $\delta$. 
Par ailleurs, par définition de la $\delta$-hyperbolicité de $X$ (\cite{harpehyper}, chap.~2, déf.~3), pour tous points~$x, y, z$ et~$t$ de~$X$, en notant $K = \text{e}^{\delta} \geq 1$, on a
\[\text{e}^{-(x,y)_t} \leq K \cdot \text{max}\left\{\text{e}^{-(x,z)_t},\text{e}^{-(z,y)_t}\right\}.\]
Un argument analogue à la proposition~10 du chapitre~7 de \cite{harpehyper} permet de démontrer qu'il existe un réel~$\epsilon > 0$ et un réel $C>1$ tels que, en tout point $t$ de~$X$, il existe une fonction $d_t$ définie sur $X \times X$ satisfaisant l'inégalité triangulaire et vérifiant, pour tous points~$x$ et~$y$ de~$X$,
\[\frac{1}{C} \cdot \text{e}^{-\epsilon \cdot (x,y)_t} \leq d_t(x,y) \leq C \cdot \text{e}^{-\epsilon \cdot (x,y)_t}.\]
Convenons d'appeler une telle fonction $d_t$ une {\it métrique visuelle} centrée en $t$, de paramètre visuel $\text{e}^{\epsilon}$ et de constante $C$.
Notons que, si $t$ appartient à $\eta\text{-}\text{géod}(x,y)$, alors les deux points~$x$ et~$y$ sont à une distance visuelle $d_t(x,y)$ l'un de l'autre au moins égale à~$\text{e}^{-\epsilon \cdot \eta/2}/C$.
Remarquons enfin que, pour tout $x$ de $X$, $d_t(x,x) > 0$, ce qui fait que $d_t$ n'est pas une métrique au sens strict du terme.

\begin{lem}\label{lemme-cle}
Soit $(X,d)$ un bon espace hyperbolique discret.
Il existe deux réels~$\epsilon$ et~$\alpha_\eta$ strictement positifs tels que pour tout chemin $(x_i)_{0 \leq i \leq n}$ de longueur $n \geq 1$, si $t$ appartient à $\eta\text{-}\text{géod}(x_0,x_n)$, alors
\[\sum_{i=0}^{n-1} \text{e}^{-\epsilon \cdot d(x_i, t)} \geq \alpha_\eta.\]
\end{lem}

\begin{proof}
Soit $(x_i)_{0 \leq i \leq n}$ un chemin de longueur $n$, $t$ un point de $\eta\text{-}\text{géod}(x_0,x_n)$ et $d_t$ une métrique visuelle sur $X$, centrée en $t$, de paramètre visuel $\text{e}^{\epsilon}$ et de constante~$C$.
On a alors les inégalités suivantes :
\begin{equation}\label{inegalite-1}
C \cdot \sum_{i=0}^{n-1} \text{e}^{-\epsilon \cdot (x_i,x_{i+1})_t} \geq \sum_{i=0}^{n-1} d_t(x_i,x_{i+1}) \geq d_t(x_0,x_n) \geq \text{e}^{-\epsilon \cdot \eta/2}/C.
\end{equation}
En effet, la première inégalité provient de la majoration satisfaite par la métrique visuelle $d_t$ et la deuxième inégalité est l'inégalité triangulaire pour cette même métrique visuelle.
Quant à la dernière inégalité, elle résulte du fait que $t$ appartient à $\eta\text{-}\text{géod}(x_0,x_n)$.

Par ailleurs, on déduit facilement du produit de Gromov et du fait que $x_i$ et $x_{i+1}$ sont deux points voisins dans $X$ l'encadrement suivant :
\begin{equation}\label{inegalite-2}
d(x_i,t)-1 \leq (x_i,x_{i+1})_t \leq d(x_i,t).
\end{equation}
En utilisant les inégalités \ref{inegalite-1} et \ref{inegalite-2} ci-dessus, on en déduit finalement l'inégalité annoncée avec $\alpha_\eta = \text{e}^{-\epsilon (\eta/2+1)}/C^2.$
\end{proof}

\begin{prop}\label{prop-principale}
Soit $(X,d)$ un bon espace hyperbolique discret.
Il existe deux réels~$\epsilon$ et~$\beta$ strictement positifs tels que pour tous points $x$ et~$y$ dans $X$ et tout chemin $(x_i)_{0 \leq i \leq n}$ tel que $x_0=x$ et $x_n=y$, on a
\[\sum_{i=0}^{n-1} \text{e}^{-\epsilon \cdot d(x_i, \text{géod}(x,y))} \geq \beta \cdot d(x,y).\]
\end{prop}

Si $X$ est un arbre, la proposition précédente est facilement vérifiée quel que soit le choix de~$\epsilon$ strictement positif avec $\beta = 1$, puisque tout chemin de~$x$ à $y$ contient le segment $\text{géod}(x,y)$.
Par contre, on vérifie sans peine que cette proposition est fausse dans le cas du plan euclidien $(\mathbf{Z}^2,d_{\text{eucl.}})$.
En effet, en notant $d$ la distance entre deux points $x$ et $y$ de $\mathbf{Z}^2$ situés sur une même horizontale et en considérant un chemin \og rectangulaire \fg\ de longueur $n=2m+d$ entre $x$ et $y$ comme indiqué sur la figure ci-dessous, on calcule facilement la distance d'un point de ce chemin au segment horizontal inférieur qui est le segment géodésique reliant $x$ et $y$.

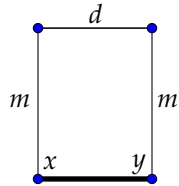
\begin{center}
\definecolor{qqqqff}{rgb}{0.,0.,1.}
\begin{tikzpicture}[line cap=round,line join=round,>=triangle 45,x=1.0cm,y=1.0cm,scale=0.5]
\clip(3.,1.) rectangle (8.,7.);
\draw [line width=2.pt] (4.,2.)-- (7.,2.);
\draw (4.,2.)-- (4.,6.);
\draw (4.,6.)-- (7.,6.);
\draw (7.,6.)-- (7.,2.);
\draw (7.,6.)-- (4.,6.);
\draw (3.9,2.85) node[anchor=north west] {$x$};
\draw (6.2,2.9) node[anchor=north west] {$y$};
\draw (5.1,6.9) node[anchor=north west] {$d$};
\draw (3,4.5) node[anchor=north west] {$m$};
\draw (6.9,4.5) node[anchor=north west] {$m$};
\begin{scriptsize}
\draw [fill=qqqqff] (4.,2.) circle (3.5pt);
\draw [fill=qqqqff] (7.,2.) circle (3.5pt);
\draw [fill=qqqqff] (4.,6.) circle (3.5pt);
\draw [fill=qqqqff] (7.,6.) circle (3.5pt);
\end{scriptsize}
\end{tikzpicture}
\end{center}

Les $d$ points du segment horizontal supérieur sont tous à distance $m$ du segment géodésique entre $x$ et $y$, et les points des segments verticaux sont à une distance de~$x$ ou~$y$ variant de $0$ à $m-1$.
On en déduit donc
\[\sum_{i=0}^{n-1} \text{e}^{-\epsilon \cdot d_{\text{eucl.}}(x_i, \text{géod}(x,y))}
= \sum_{k=0}^{m-1} \text{e}^{-\epsilon \cdot k} + d \cdot \text{e}^{-\epsilon \cdot m} + \sum_{k=1}^{m} \text{e}^{-\epsilon \cdot k}
\leq \frac{2}{1-\text{e}^{-\epsilon}} + d \cdot \text{e}^{-\epsilon \cdot m},\]

\noindent
ce qui interdit à la somme considérée dans la proposition~\ref{prop-principale} d'être minorée par $\beta \cdot d$ pour tous $d$ et $m$, dès lors que $\beta$ et $\epsilon$ sont strictement positifs.

\begin{proof}
Soit $x$ et $y$ deux points de $X$ et $(x_i)_{0 \leq i \leq n}$ un chemin de $x$ à $y$.


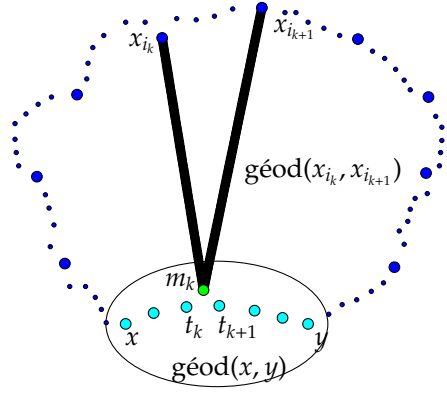
\begin{center}
\definecolor{qqffqq}{rgb}{0.,1.,0.}
\definecolor{qqqqff}{rgb}{0.,0.,1.}
\definecolor{qqffff}{rgb}{0.,1.,1.}
\begin{tikzpicture}[line cap=round,line join=round,>=triangle 45,x=1.0cm,y=1.0cm,scale=0.8]
\clip(1.8,0.7) rectangle (9.5,7.5);
\draw (3.84,2.02) node[anchor=north west] {$x$};
\draw (6.94,2.02) node[anchor=north west] {$y$};
\draw [rotate around={0.:(5.5,2.)}] (5.5,2.) ellipse (1.8251407699364384cm and 1.0397782600555683cm);
\draw (3.9,6.98) node[anchor=north west] {$x_{i_k}$};
\draw (6.28,7.2) node[anchor=north west] {$x_{i_{k+1}}$};
\draw (4.6,1.62) node[anchor=north west] {$\text{géod}(x,y)$};
\draw (4.82,2.28) node[anchor=north west] {$t_k$};
\draw (5.36,2.28) node[anchor=north west] {$t_{k+1}$};
\draw [line width=3.6pt,dash pattern=on 1pt off 1pt] (4.6,6.74)-- (5.28,2.56);
\draw [line width=3.6pt,dash pattern=on 1pt off 1pt] (5.28,2.56)-- (6.24,7.24);
\draw (4.5,3) node[anchor=north west] {$m_k$};
\draw (5.8,4.9) node[anchor=north west] {$\text{géod}(x_{i_k},x_{i_{k+1}})$};
\begin{scriptsize}
\draw [fill=qqffff] (4.,2.) circle (2.5pt);
\draw [fill=qqffff] (7.,2.) circle (2.5pt);
\draw [fill=qqqqff] (3.,3.) circle (2.5pt);
\draw [fill=qqqqff] (6.24,7.24) circle (2.5pt);
\draw [fill=qqqqff] (2.54,4.44) circle (2.5pt);
\draw [fill=qqqqff] (3.2,5.8) circle (2.5pt);
\draw [fill=qqqqff] (4.6,6.74) circle (2.5pt);
\draw [fill=qqqqff] (7.82,6.72) circle (2.5pt);
\draw [fill=qqqqff] (8.98,5.76) circle (2.5pt);
\draw [fill=qqqqff] (8.94,4.5) circle (2.5pt);
\draw [fill=qqqqff] (8.42,2.94) circle (2.5pt);
\draw [fill=qqqqff] (2.18,5.32) circle (1.0pt);
\draw [fill=qqqqff] (2.36,5.48) circle (1.0pt);
\draw [fill=qqqqff] (2.56,5.64) circle (1.0pt);
\draw [fill=qqqqff] (2.78,5.76) circle (1.0pt);
\draw [fill=qqqqff] (2.18,5.06) circle (1.0pt);
\draw [fill=qqqqff] (2.2,4.84) circle (1.0pt);
\draw [fill=qqqqff] (2.24,4.62) circle (1.0pt);
\draw [fill=qqqqff] (2.68,4.14) circle (1.0pt);
\draw [fill=qqqqff] (2.74,3.92) circle (1.0pt);
\draw [fill=qqqqff] (2.82,3.64) circle (1.0pt);
\draw [fill=qqqqff] (2.76,3.28) circle (1.0pt);
\draw [fill=qqqqff] (3.08,2.62) circle (1.0pt);
\draw [fill=qqqqff] (3.34,2.62) circle (1.0pt);
\draw [fill=qqqqff] (3.5,2.44) circle (1.0pt);
\draw [fill=qqqqff] (3.6,2.24) circle (1.0pt);
\draw [fill=qqqqff] (3.68,2.02) circle (1.0pt);
\draw [fill=qqqqff] (3.36,6.06) circle (1.0pt);
\draw [fill=qqqqff] (3.38,6.32) circle (1.0pt);
\draw [fill=qqqqff] (3.46,6.58) circle (1.0pt);
\draw [fill=qqqqff] (3.64,6.78) circle (1.0pt);
\draw [fill=qqqqff] (4.,7.) circle (1.0pt);
\draw [fill=qqqqff] (4.32,7.) circle (1.0pt);
\draw [fill=qqqqff] (4.84,7.06) circle (1.0pt);
\draw [fill=qqqqff] (5.16,7.04) circle (1.0pt);
\draw [fill=qqqqff] (5.46,7.1) circle (1.0pt);
\draw [fill=qqqqff] (5.76,7.22) circle (1.0pt);
\draw [fill=qqqqff] (6.62,7.16) circle (1.0pt);
\draw [fill=qqqqff] (6.8,7.16) circle (1.0pt);
\draw [fill=qqqqff] (7.,7.) circle (1.0pt);
\draw [fill=qqqqff] (7.26,6.9) circle (1.0pt);
\draw [fill=qqqqff] (7.52,6.78) circle (1.0pt);
\draw [fill=qqqqff] (8.06,6.52) circle (1.0pt);
\draw [fill=qqqqff] (8.26,6.4) circle (1.0pt);
\draw [fill=qqqqff] (8.44,6.22) circle (1.0pt);
\draw [fill=qqqqff] (8.58,6.04) circle (1.0pt);
\draw [fill=qqqqff] (8.68,5.88) circle (1.0pt);
\draw [fill=qqqqff] (9.1,5.44) circle (1.0pt);
\draw [fill=qqqqff] (9.2,5.26) circle (1.0pt);
\draw [fill=qqqqff] (9.2,5.04) circle (1.0pt);
\draw [fill=qqqqff] (9.18,4.84) circle (1.0pt);
\draw [fill=qqqqff] (9.16,4.62) circle (1.0pt);
\draw [fill=qqqqff] (8.96,4.16) circle (1.0pt);
\draw [fill=qqqqff] (8.76,4.08) circle (1.0pt);
\draw [fill=qqqqff] (8.68,3.9) circle (1.0pt);
\draw [fill=qqqqff] (8.64,3.7) circle (1.0pt);
\draw [fill=qqqqff] (8.5,3.54) circle (1.0pt);
\draw [fill=qqqqff] (8.58,3.32) circle (1.0pt);
\draw [fill=qqqqff] (8.24,2.68) circle (1.0pt);
\draw [fill=qqqqff] (8.04,2.54) circle (1.0pt);
\draw [fill=qqqqff] (7.8,2.4) circle (1.0pt);
\draw [fill=qqqqff] (7.56,2.32) circle (1.0pt);
\draw [fill=qqqqff] (7.34,2.16) circle (1.0pt);
\draw [fill=qqffff] (4.46,2.18) circle (2.5pt);
\draw [fill=qqffff] (5.,2.28) circle (2.5pt);
\draw [fill=qqffff] (5.54,2.3) circle (2.5pt);
\draw [fill=qqffff] (6.12,2.22) circle (2.5pt);
\draw [fill=qqffff] (6.58,2.1) circle (2.5pt);
\draw [fill=qqffqq] (5.28,2.56) circle (2.5pt);
\end{scriptsize}
\end{tikzpicture}
\end{center}

Pour tout $i$ de $[[0,n-1]]$, la distance entre tout point de $\text{géod}(x,y)$ parmi les plus proches de $x_i$ et tout point de $\text{géod}(x,y)$ parmi les plus proches de $x_{i+1}$ est inférieure à $\kappa\cdot\delta$, où $\kappa$ est une constante ne dépendant que de $(X,\delta)$.
Par conséquent il est possible de fixer deux constantes $\delta_1$ et $\delta_2$ (ne dépendant que de $\delta$) vérifiant $\delta_2 \geq \delta_1 > 3\delta+1$ telles que si la distance entre $x$ et $y$ est supérieure à $\delta_1$, on peut extraire une sous-suite $(x_{i_k})_{0 \leq k \leq l}$ du chemin $(x_i)_{0 \leq i \leq n}$ et une suite de points $(t_k)_{0 \leq k \leq l}$ de $\text{géod}(x,y)$ telles que $x_{i_0}=t_0=x$, $x_{i_{l}}=t_{l}=y$, et pour tout $k$ dans $[[0,l-1]]$ :
\begin{itemize}
\renewcommand{\labelitemi}{$\bullet$}
\item $t_k$ est un point de $\text{géod}(x,y)$ parmi les plus proches de $x_{i_k}$ (autrement dit, $d(x_{i_k},t) \geq d(x_{i_k},t_k)$ pour tout point $t$ de $\text{géod}(x,y)$) ;
\item la distance entre $t_k$ et $t_{k+1}$ est contrôlée par $\delta_1$ et $\delta_2$ : $\delta_1 \leq d(t_k,t_{k+1}) \leq \delta_2$.
\end{itemize}
Si la distance entre~$x$ et $y$ est strictement inférieure à $\delta_1$, la proposition résulte du lemme~\ref{lemme-cle}.
Sinon on applique l'argument suivant.

Soit $m_k$ un milieu\footnote{Convenons de désigner par milieu d'une paire de points $\{x,y\}$ de $X$ un point $z$ de $\geod(x,y)$ tel que $| d(x,z)-d(y,z) | \leq 1.$} de la paire de points $\{t_k,t_{k+1}\}$ pour laquelle on note $D=d(t_k,t_{k+1})$.
Une première application du lemme~\ref{lemme-emboitement} d'emboîtement des géodésiques au quadruplet de points $(x_{i_k},t_k,m_k,t_{k+1})$ (avec $a=x_{i_k}, b=t_k, c=m_k, d=t_{k+1}, \eta_1=d(t_k,m_k), \eta_2=0$) permet d'affirmer que $t_k$ appartient à $\Delta\text{-}\geod(x_{i_k},t_{k+1})$ avec $\Delta=(D+1)/2+\delta$, puisque
\[d(t_k,m_k) > \frac{\eta_1+\eta_2+\delta}{2} = d(t_k,m_k)/2+\delta/2,\]
ce qui suit de $\delta_1>\delta$.
De manière symétrique, une deuxième application du même lemme au quadruplet $(t_k,m_k,t_{k+1},x_{i_{k+1}})$ assure cette fois-ci que $t_{k+1}$ appartient à $\Delta\text{-}\geod(t_k,x_{i_{k+1}})$.

Une troisième et dernière application du lemme~\ref{lemme-emboitement} au quadruplet $(x_{i_k},t_k,t_{k+1},x_{i_{k+1}})$ permet de déduire que $t_k$ et $t_{k+1}$ appartiennent à $\Delta'\text{-}\geod(x_{i_k},x_{i_{k+1}})$ avec $\Delta'=\Delta+\delta$, puisque
\[d(t_k,t_{k+1}) = D > \frac{2\Delta+\delta}{2}=D/2+(3\delta+1)/2,\]
ce qui suit de $\delta_1>3\delta+1$.

On en déduit finalement que $m_k$ appartient à $\Delta''\text{-}\geod(x_{i_k},x_{i_{k+1}})$ avec $\Delta''=\Delta'+D+1=3(D+1)/2+2\delta$.
En effet,
\[d(x_{i_k},m_k)+d(m_k,x_{i_{k+1}}) \leq d(x_{i_k},t_k)+d(t_k,x_{i_{k+1}})+2d(t_k,m_k) \leq d(x_{i_k},x_{i_{k+1}})+\Delta''.\]

Or d'après le lemme~\ref{lemme-cle} (avec $x=x_{i_k}, y=x_{i_{k+1}}, t=m_k, \eta=\Delta''$), il existe deux réels~$\epsilon$ et~$\alpha$ strictement positifs tels que, pour tout $k$,
\[\sum_{i=i_k}^{i_{k+1}-1} \text{e}^{-\epsilon \cdot d(x_i, \text{géod}(x,y))} \geq
\sum_{i=i_k}^{i_{k+1}-1} \text{e}^{-\epsilon \cdot d(x_i, m_k)} \geq \alpha.\]
On en déduit finalement l'inégalité annoncée en posant $\beta = \alpha/{\delta_2}$ puisqu'alors
\[\sum_{i=0}^{n-1} \text{e}^{-\epsilon \cdot d(x_i, \text{géod}(x,y))}
= \sum_{k=0}^{l-1} \sum_{i=i_k}^{i_{k+1}-1} \text{e}^{-\epsilon \cdot d(x_i, \text{géod}(x,y))}
\geq \alpha \cdot d(x,y)/{\delta_2}.\]
\end{proof}

On appelle \textit{chaîne} toute combinaison linéaire (finie) $c$ d'arêtes orientées, avec la relation que l'arête opposée $e^{\text{op}}$ d'une arête orientée $e$ est égale à $-e$.
Autrement dit, une chaîne est un élément du quotient de l'espace vectoriel libre engendré par les arêtes orientées modulo la relation $e^{\text{op}}+e=0$ pour toute arête orientée $e$ ; on note $\overline{e}$ l'image dans le quotient d'une telle arête orientée $e$.
On note $c(e)$ la différence des coefficients de $\overline{e}$ et $\overline{e^{\text{op}}}$ dans $c$, de sorte que $c \mapsto c(e)$ est la forme linéaire sur ce quotient qui envoie $\overline{e}$ sur 1, $\overline{e^{\text{op}}}$ sur -1, et les autres arêtes orientées sur 0.
On remarque que $c(e) \cdot \overline{e}$ ne dépend pas de l'orientation de $e$ (car $c(e)$ et $\overline{e}$ changent tous les deux de signe quand on remplace $e$ par $e^{\text{op}}$).
On a alors $c = \sum_{e/\sim} c(e) \cdot \overline{e}$, où la somme porte sur les arêtes modulo l'orientation.
Enfin, nous dirons qu'une arête $e$ est dans le support de $c$ si $c(e)$ est non nul.

\begin{cor}\label{cor-prop-principale}
Pour toute chaîne $c = \sum_{e/\sim} c(e) \cdot \overline{e}$ de bord $\delta_y - \delta_x$, on a
\[\sum_{e/\sim} |c(e)| \text{e}^{-\epsilon \cdot d(e, \text{géod}(x,y))} \geq \beta \cdot d(x,y),\]
où $d(e, \text{géod}(x,y)) = \text{min}(d(e^-, \text{géod}(x,y)),d(e^+, \text{géod}(x,y)))$.
\end{cor}

Pour démontrer le corollaire~\ref{cor-prop-principale}, nous aurons besoin du lemme suivant.
Nous dirons qu'une chaîne est un \textit{cycle} si son bord est nul.

\begin{lem}\label{chaine-somme-chemins-lacets}
Pour toute chaîne $c = \sum_{e/\sim} c(e) \cdot \overline{e}$ de bord $\delta_y - \delta_x$, il existe des coefficients réels $\alpha_k$ non nuls, des chemins $c_k$ de bord $\delta_y - \delta_x$ et un cycle $\ell = \sum_j \beta_j \cdot l_j$, où les $\beta_j$ sont des coefficients strictement positifs et les $l_j$ des lacets, tels que $c = \sum_k \alpha_k \cdot c_k + \ell$, $\sum_k \alpha_k = 1$ et
\[ || c ||_1 = \sum_{e/\sim} | c(e) | = \sum_k |\alpha_k| \cdot ||c_k||_1 + || \ell ||_1.\]
\end{lem}

\begin{proof}
Soit $c = \sum_{e/\sim} c(e) \cdot \overline{e}$ une chaîne de bord $\delta_y - \delta_x$.
On note $\mathcal{S}_c$ le graphe orienté formé des arêtes $e$ telles que $c(e)$ est strictement positif : c'est un raffinement du support de $c$ tenant compte de l'orientation.

Si c est nulle, il n’y a rien à démontrer.
On suppose donc $c$ non nulle.

Puisque le bord de~$c$ est $\delta_y - \delta_x$, pour tout sommet~$z$ distinct de~$x$ ou~$y$, la somme des coefficients $c(e)$ des arêtes~$e$ de sommet terminal~$z$ dans le graphe~$\mathcal{S}_c$ est égale à la somme des coefficients $c(e)$ des arêtes~$e$ de sommet origine~$z$\footnote{Si on pense à $\mathcal{S}_c$ comme un circuit électrique, on traduit ainsi mathématiquement la loi des n{\oe}uds de Kirchhoff qui assure qu'en tout n{\oe}ud du réseau, la somme des intensités des courants entrants est égale à la somme des intensités des courants sortants.}.
En particulier, si $e$ est une arête de $\mathcal{S}_c$ de sommet terminal $z$ distinct de~$x$ ou~$y$, alors il existe une arête~$e'$ dans $\mathcal{S}_c$ de sommet origine $z$.
Ainsi, partant du sommet~$x$, on construit, arête après arête, un chemin $c_1$ d'origine $x$ et d'extrémité~$y$ dans~$\mathcal{S}_c$, ou bien un lacet $l_1$ si le chemin se recoupe.

Selon le cas, on note $\alpha_1 >0$ (respectivement $\beta_1$) le minimum des $c_1(e)$ pour les arêtes $e$ dans le support de $c_1$ (resp. $l_1$).
On considère alors la chaîne $c' = c - \alpha_1 \cdot c_1$ (resp. $c - \beta_1 \cdot l_1$) de bord $(1-\alpha_1) \cdot (\delta_y - \delta_x)$ (resp. de même bord que $c$) et de norme~$\ell^1$ égale à $||c||_1 - \alpha_1 \cdot ||c_1||_1$ (resp. $||c||_1 - \beta_1 \cdot || l_1 ||_1$).
Par construction, le cardinal du support de $c'$ est strictement inférieur à celui de $c$.
Si le bord de $c'$ est nul, c'est terminé.
Sinon on réitère le raisonnement précédent avec la chaîne $\frac{1}{1-\alpha_1}c'$ (resp. $c'$) au lieu de $c$.
On en déduit le résultat après un nombre fini d'itérations.
\end{proof}

Nous démontrons à présent le corollaire~\ref{cor-prop-principale}.

\begin{proof}
Soit $c = \sum_{e/\sim} c(e) \cdot \overline{e}$ une chaîne de bord $\delta_y - \delta_x$.
D'après le lemme~\ref{chaine-somme-chemins-lacets}, il existe des chemins $c_k$ de bord $\delta_y - \delta_x$, des réels non nuls $\alpha_k$ et un cycle~$\ell$ tels que
\[c= \sum_k \alpha_k \cdot c_k + \ell,\]
\noindent
avec $\sum_k \alpha_k = 1$ et $|| c ||_1 = \sum_k |\alpha_k| \cdot ||c_k||_1 + || \ell ||_1$.
Pour chaque chemin $c_k$ de bord $\delta_y - \delta_x$, en indexant par $e_{k,j}$ les arêtes d'un tel chemin, la proposition~\ref{prop-principale} assure que
\[\sum_j \text{e}^{-\epsilon \cdot d(e_{k,j}^{-}, \text{géod}(x,y))}
\geq \beta \cdot d(x,y).\]
On en déduit donc
\[\sum_{e/\sim} |c(e)| \text{e}^{-\epsilon \cdot d(e, \text{géod}(x,y))}
\geq \sum_k |\alpha_k| \sum_j \text{e}^{-\epsilon \cdot d(e_{k,j}^{-}, \text{géod}(x,y))}
\geq \beta \cdot d(x,y).\]
\end{proof}

\begin{prop}\label{prop2}
Soit $(X,d)$ un bon espace hyperbolique discret.
Pour tout $p$ suffisamment proche de 1, il existe un réel $\alpha'$ strictement positif tel que pour toute chaîne $c = \sum_{e/\sim} c(e) \cdot \overline{e}$ de bord $\delta_y - \delta_x$, on a
\[||c||_p = \left( \sum_{e/\sim} |c(e)|^p \right)^{1/p} \geq \alpha' \cdot d(x,y)^{1/p}.\]
\end{prop}

\begin{proof}
Soit $1<p<\infty$ et $q$ tel que $1/p+1/q=1$.
L'inégalité de Hölder assure que pour tous $f$ de $\ell^p$ et $g$ de $\ell^q$, la norme $||\cdot||_1$ de $fg$ est majorée par le produit $||f||_p ||g||_q$.
On en déduit donc
\[\left( \sum_{e/\sim} |c(e)|^p \right)^{1/p} \left( \sum_{e/\sim} \text{e}^{-\epsilon q \cdot d(e, \text{géod}(x,y))} \right)^{1/q}
\geq \sum_{e/\sim} |c(e)| \text{e}^{-\epsilon \cdot d(e, \text{géod}(x,y))}
\geq \beta \cdot d(x,y),\]
\noindent
la première inégalité résultant de l'inégalité de Hölder et la deuxième du corollaire~\ref{cor-prop-principale}.
Or, l'espace $(X,d)$ étant uniformément localement fini, le nombre de sommets à distance $k$ de $\text{géod}(x,y)$ est au plus à croissance exponentielle en~$k$.
Plus précisément, il existe un réel~$\beta'$ strictement positif et une constante multiplicative~$D$ tels que, pour $x$ distinct de~$y$,
\[\text{card} (\{z \in X \, ; \, d(z, \text{géod}(x,y)) \leq k \}) \leq D \cdot \text{e}^{\beta' k} d(x,y),\]
\noindent
ce qui permet de donner la minoration suivante
\[\left( \sum_{e/\sim} \text{e}^{-\epsilon q \cdot d(e, \text{géod}(x,y))} \right)^{1/q}
\leq D^{1/q} \left( \sum_{k=0}^{\infty} \text{e}^{(\beta'-\epsilon q) k} \right)^{1/q} d(x,y)^{1/q}.\]
La série $\sum_k \text{e}^{(\beta'-\epsilon q) k}$ ci-dessus est convergente si et seulement si $q>\beta'/\epsilon$.
Dans ce cas, en notant $D'$ la somme de cette série, on en déduit l'énoncé du lemme pour tout $1 < p < \beta'/(\beta'-\epsilon)$ en posant $\alpha' = \beta/(D D')^{1/q}$.
\end{proof}

Ceci termine la démonstration du théorème~\ref{thm-principal} puisque, étant donné un point~$o$ de $X$, l'action de $G$ sur $X$ étant propre, la proposition~\ref{prop2} assure alors que le cocycle $g \longmapsto \delta_{o} - \delta_{g\cdot o} $ à valeurs dans $\partial(\ell^p(X^1))$ est propre pour tout $p$ suffisamment proche de 1 (la continuité du cocycle résultant du fait qu'il est localement constant, comme conséquence de la continuité de $g \longmapsto g \cdot o$ de $G$ dans $X$ discret).
\end{proof}

\vspace{0.5cm}

\noindent
{\bf Remerciements :}
Nous remercions Mikael de la Salle pour de stimulantes discussions et Peter Haissinsky pour des références concernant la métrique visuelle.


\end{document}